\documentclass[reqno, 12pt]{amsart}
\usepackage{amsmath,amssymb,amsthm}    
\usepackage{mathabx}    
\usepackage{mathrsfs}    	
\usepackage[backref,colorlinks]{hyperref} 
\usepackage[abbrev,msc-links]{amsrefs}  
\usepackage{enumerate}
\usepackage{enumitem}
\usepackage[symbol]{footmisc}
\usepackage[symbol]{footmisc}
\usepackage{pgf,tikz}
\usepackage{float}
\usepackage{mathrsfs}
\usepackage{caption}
\usepackage{comment}
\usetikzlibrary{arrows}

\normalsize                              
\newtheorem{theorem}{Theorem}[section]
\newtheorem{definition}[theorem]{Definition}
\newtheorem{lemma}[theorem]{Lemma}
\newtheorem{claim}[theorem]{Claim}

\newtheorem{fact}[theorem]{Fact}

\newtheorem{proposition}[theorem]{Proposition}
\newtheorem{corollary}[theorem]{Corollary}
\newtheorem{note}[theorem]{Note}
\newtheorem*{theorem*}{Theorem}

\newtheorem{innercustomthm}{Theorem}
\newenvironment{customthm}[1]
  {\renewcommand\theinnercustomthm{#1}\innercustomthm}
  {\endinnercustomthm}

\newtheorem{innercustomprop}{Proposition}
\newenvironment{customprop}[1]
  {\renewcommand\theinnercustomprop{#1}\innercustomprop}
  {\endinnercustomprop}

\def\dHH#1{\leavevmode\setbox0=\hbox{#1}\dimen0=\wd0\setbox0=\hbox{.}%
	\advance\dimen0 by -\wd0%
	\hbox{#1\raise-0.5ex\hbox to 0pt{\hss.\kern.5\dimen0}}}%

\newcommand{\nn}{{\mathbb{N}}}

\newcommand{\aar}[1]{{\color{blue} #1}}

\newcommand{\ber}[1]{{\color{red} #1}}

\begin{document}
\title[Increasing paths in graphs]{Increasing Paths in countable graphs}

\author[A.Arman]{Andrii Arman}
\address{School of mathematics, Monash University, Melbourne, VIC, Australia}
\email{Andrii.Arman@monash.edu}
\thanks{}

\author[B. Elliott]{Bradley Elliott}
\address{Department of Mathematics, 
Emory University, Atlanta, GA 30322, USA}
\email{bradley.elliott@emory.edu}

\author[V. R\"{o}dl]{Vojt\v{e}ch R\"{o}dl}
\address{Department of Mathematics, 
Emory University, Atlanta, GA 30322, USA}
\email{rodl@mathcs.emory.edu}
\thanks{The third  author was supported by NSF grant DMS 1764385}




\begin{abstract} 
In this paper we study variations of an old result by M\"{u}ller, Reiterman, and the last author stating that a countable graph has a subgraph with infinite degrees if and only if in any labeling of the vertices (or edges) of this graph by positive integers we can always find an infinite increasing path. We study corresponding questions for hypergraphs and directed graphs. For example we show that the condition that a hypergraph contains a subhypergraph with infinite degrees is equivalent to the condition that any vertex labeling contains an infinite increasing loose path. We also find an equivalent condition for a graph to have a property that any vertex labeling with positive integers contains a path of arbitrary finite length, and we study related problems for oriented graphs and labelings with $\mathbb{Z}$ (instead of $\mathbb{N}$). For example, we show that for every simple hypergraph, there is a labelling of its edges by $\mathbb{Z}$ that forbids one-way infinite increasing paths.
\end{abstract}

\maketitle

\section{Introduction}

For a countable graph $G=(V,E)$ a {\it labeling} (or ordering) of vertices by $\mathbb{N}$ is a bijection from $V$ to $\nn$ and a labeling of edges is a bijection from $E$ to $\nn$. Given a countable graph $G=(V,E)$ and a labeling $\phi$ of $V$ by $\mathbb{N}$, we say $G$ contains an \emph{infinite increasing path} under $\phi$ if there exists an infinite path of vertices $\{v_i\}_{i=1}^\infty$ in $G$ such that $\phi(x_i) < \phi(x_{i+1})$ for all $i\geq 1$. Similarly, given a countable graph $G=(V,E)$ and a labeling $\phi$ of $E$ by $\mathbb{N}$, we say $G$ contains an infinite increasing path under $\phi$ if there exists an infinite path of edges $\{e_i\}_{i=1}^\infty$ in $G$ such that $\phi(e_i) < \phi(e_{i+1})$ for all $i\geq 1$.

In 1982, M\"uller and R\"odl~\cite{MR:82} showed that a countable graph $G$ contains an infinite increasing path in any labeling of vertices by $\mathbb{N}$ if and only if $G$ contains a subgraph with infinite degrees. In their paper M\"uller and R\"odl asked whether or not the condition of having a subgraph with infinite degrees is also necessary for containing an infinite path in any edge labeling of $G$ by $\mathbb{N}$. This was confirmed to be true by Reiterman~\cite{Reit:89} in 1989. Together these two results are formulated below.
\begin{theorem}[\cite{MR:82},\cite{Reit:89}]\label{RM}
Let $G$ be a countable graph.  Then the following are equivalent:
\begin{itemize}
    \item[$(1)$] $G$ contains a subgraph $G^\prime$, such that any $v\in V(G^\prime)$ has infinite degree in $G^\prime$.
    \item[$(2)$] any labeling $\phi$ of the vertices of $G$ with $\mathbb{N}$ contains an infinite increasing path.
    \item[$(3)$] any labeling $\phi$ of the edges of $G$ with $\mathbb{N}$ contains an infinite increasing path.
\end{itemize}
\end{theorem}

Motivated by Theorem~\ref{RM}, we study possible ways to extend this result. In Section~\ref{sec:hyper} we study this result for simple hypergraphs and loose paths, in Section~\ref{sec:FIN} we find an equivalent condition for a graph to contain an increasing path of arbitrary finite length under any vertex labeling, and finally in Section~\ref{sec:directed} we study labelings with $\mathbb{Z}$ and possible extensions to oriented graphs.  In section~\ref{sec:problems} we state some open problems.
\vspace{0.5cm}

\textsc{Hypergraphs.} Our main motivation for Section~\ref{sec:hyper} is to find a hypergraph extension of Theorem~\ref{RM}. 
We focus solely on simple (also known as linear) hypergraphs in this paper; these are hypergraphs in which each pair of vertices share at most one edge. There exist many examples of non-simple hypergraphs in which every vertex has infinite degree but which do not contain an infinite path. For example, take a complete graph on countably many vertices, and extend each edge by including the same new vertex. The theorems below, however, still hold if ``simple" is replaced by ``each pair of vertices have finite co-degree. Because all hypergraphs considered here are simple, all paths in this section will be loose paths.

\begin{definition}\label{def:infincloosepath}
Given a $k$-uniform hypergraph $H=(V,E)$ and a labeling $\phi$ of $V$ by $\mathbb{N}$, an \emph{infinite increasing path} in $H$ is an infinite path with vertex set $\{v_j: j\geq 1\}$ and edges $e_0,\ldots, e_i, \ldots$, where 
\begin{enumerate}
\item[$(1)$] each $e_i=\{v_{(k-1)i+1}, v_{(k-1)i+2}, \ldots, v_{(k-1)i+k}\}$ for $i\geq 0$,
\item[$(2)$] all $v_j$-s are distinct, and
\item[$(3)$] $\phi(v_{j})< \phi(v_{j+1})$ for $j\geq 1$.
\end{enumerate}
\end{definition}
In the graph case, the condition of containing a subgraph with infinite degrees implies existence of an infinite increasing path in any vertex labeling and this condition is a natural candidate for the hypergraph case; however, the notion of a subgraph with infinite degrees has multiple meanings in the hypergraph case.

\begin{definition}\label{def:C_l}
Let $\ell\leq k$. We say that a $k$-uniform $H=(V,E)$ has property $C_\ell$ if there is a set $V' \subseteq V$ such that $\forall v \in V'$, $|\{e\in E: v\in e\; \text{and} \; |V'\cap e|\geq \ell\}| = \infty$.
\end{definition}
For a $k$-uniform hypergraph $H$, property $C_k$ is perhaps the most natural hypergraph extension of the graph condition, and is equivalent to $H$ containing an induced subhypergraph in which every vertex has infinite degree. It turns out property $C_k$ implies more than just containing an infinite increasing path under any vertex labeling of $V(H)$. It implies existence of what we call an infinite $(k-1)$-branching branching tree.

\begin{definition}\label{def:infbranchingtree}
\begin{enumerate}
\item[]
\item[$(1)$] The \emph{infinite $(k-1)$-branching tree is the unique $k$-uniform} hypergraph so that each pair of vertices is connected by a unique path and so that all vertices have degree 2, except the root vertex, which has degree 1. See figure~\ref{fig:tree}.
\item[$(2)$] In a rooted infinite $(k-1)$-branching tree, $t(v)$ will denote the number of edges on the path from vertex $v$ to the root.
\item[$(3)$] Given a $k$-uniform hypergraph $H=(V,E)$ and a labeling $\phi$ of $V$ by $\mathbb{N}$, an \emph{infinite increasing {$(k-1)$- branching} tree} in $H$ is an infinite {$(k-1)$-branching} tree so that for any $x,y\in V$ with $t(x) < t(y)$, $\phi(x) < \phi(y)$.
\end{enumerate}
\end{definition}

\begin{figure}[H]
              \includegraphics[width=0.5\linewidth]{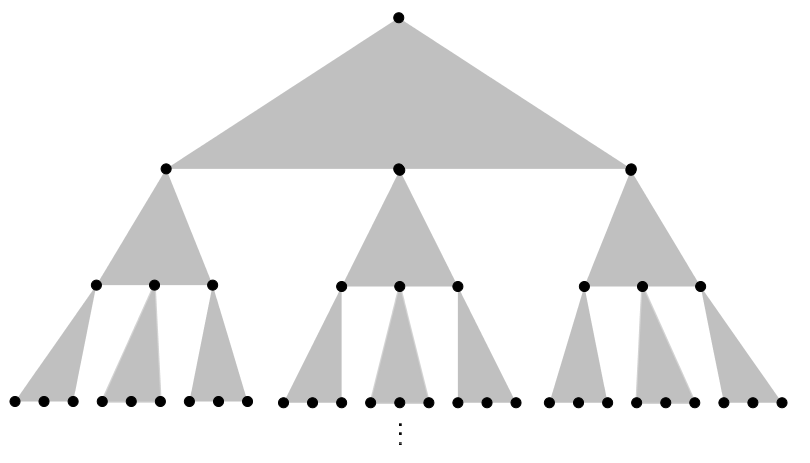}
              \caption{An example of an infinite 3-branching tree, which is a 4-uniform hypergraph. \label{fig:tree}}
          \end{figure}

In Section~\ref{sec:hyper} we consider conditions that imply the existence of infinite increasing paths and the infinite $(k-1)$-branching trees under vertex- and edge-labelings of simple hypergraphs.

\begin{customthm}{\ref{infbranchingtree}}
For a countable, $k$-uniform hypergraph $H= (V,E)$ the following are equivalent
\begin{enumerate}
    \item[$(1)$] $H$ has property $C_k$.
    \item[$(2)$] For any labeling of $V$ with $\mathbb{N}$ there exists an infinite increasing {$(k-1)$-branching} tree in $H$.
    \item[$(3)$] For any labeling of $V$ with $\mathbb{N}$ there exists an infinite increasing path in $H$.
\end{enumerate}
\end{customthm}

For edge-labelings of $k$-uniform hypergraphs, infinite increasing paths and $(k-1)$-branching trees are defined similarly:

\begin{definition}\label{def:infincloosepathE}
Given a $k$-uniform hypergraph $H=(V,E)$ and a labeling $\phi$ of $E$ by $\mathbb{N}$, an \emph{infinite increasing path} in $H$ is an infinite  path $e_1, \ldots, e_i, \ldots$ such that $\{\phi(e_i)\}_{i=1}^\infty$ is increasing.
\end{definition}

\begin{definition}\label{def:infbranchingtreeE}
\begin{enumerate}
\item[]
\item[$(1)$] In a rooted infinite {$(k-1)$-branching} tree, let $t(e)$ be the number of edges on the path from edge $e$ to the root.  E.g. $t(e)=1$ for the edge $e$ that contains the root\ber{.}
\item[$(2)$] Given a labeling $\phi$ of $E$ by $\mathbb{N}$, an \emph{infinite increasing $(k-1)$-branching tree} in $H$ is an infinite {$(k-1)$-branching} tree so that for any $e_1, e_2 \in E$ with $t(e_1) < t(e_2)$, $\phi(e_1) < \phi(e_2)$.
\end{enumerate}
\end{definition}

It turns out that property $C_2$ is a sufficient condition for containing an infinite increasing path under any edge labeling and a necessary condition for containing an infinite increasing $(k-1)$-branching tree.

\begin{customthm}{\ref{thm:hyperedge}}
Let $H= (V,E)$ be a simple, countable, $k$-uniform hypergraph. From the following conditions $\emph{(1)}\Rightarrow \emph{(2)} \Rightarrow \emph{(3)}$.
\begin{enumerate}
    \item[$(1)$] For any labeling $\phi$ of $E$ by $\mathbb{N}$ there is an infinite increasing {$(k-1)$-branching} tree in $H$.
    \item[$(2)$] $H$ has property $C_2$.
    \item[$(3)$] For any labeling $\phi$ of $E$ by $\mathbb{N}$ there is an infinite increasing path in $H$.
\end{enumerate}
\end{customthm}

 Unlike in the vertex-labeling case, with edge-labelings it is not true that the three conditions above are equivalent.  In particular we will observe that ${(2)}$ does not imply ${(1)}$, but we do not know whether ${(2)}$ is equivalent to ${(3)}$. 

\textbf{Question.} \textit{Is it true that $(2) \Leftrightarrow (3)$?}
\vspace{0.5cm}

\textsc{Paths of finite length.} In~\cite{MR:82}, the authors also showed for a graph $G$=(V,E) that for any well-ordered set $\mathcal{L}$ and any labeling of $V$ by $\mathcal{L}$ there exists an arbitrarily long increasing path if and only if the chromatic number of $G$ is infinite.  The main theorem in Section~\ref{sec:FIN} is about restricting $\mathcal{L}$ to $\mathbb{N}$, i.e. it is about finding an equivalent condition for a graph $G$ to contain an increasing path of arbitrary finite length under any labeling of vertices by $\mathbb{N}$. Theorem~\ref{RM} implies that if $G$ has a subgraph with infinite degrees, then in any labeling of vertices by $\nn$ we can find a infinite increasing path. Hence the only interesting case to study is when $G$ does not have a subgraph with infinite degrees.

\begin{definition}\label{def:fin}
We say that a countable graph $G$ has a \emph{property FIN} if for any labeling of vertices of $G$ with $\mathbb{N}$ and any $k\in \mathbb{N}$ there exists an increasing path of length $k$, but there exists of labeling of $G$ with no infinite increasing path.
\end{definition}

\begin{definition}\label{def:chistar}
We define $\chi^*(G)$ to be the minimum $k$, so that $V(G)$ can be partitioned into $k$ classes
$V_1,\ldots ,V_k$ in such a way that 
\begin{enumerate}
\item[$(1)$] each $V_j$ is an independent set, and
\item[$(2)$] each vertex in $V_j$ has finite degree into $V_{1}, V_{2},\ldots ,V_{j-1}$.
\end{enumerate}
If no such $k$ exists, set $\chi^*(G)=\infty$. Clearly, $\chi^*(G)\geq \chi(G)$.
\end{definition}

The main result of Section~\ref{sec:FIN} is

\begin{customthm}{\ref{thm:FIN}}
A graph $G$ has property FIN if and only if $\chi^*(G)=\infty$ and $G$ does not have a subgraph in which all degrees are infinite.
\end{customthm}
We show that $\chi^*$ cannot be changed to $\chi$ in the statement of the theorem above by constructing a bipartite graph $G$ that satisfies FIN. So for property FIN, the chromatic number of a graph has insignificant impact on the existence of an infinite increasing path when $V(G)$ is labeled with $\mathbb{N}$, in contrast to the case when the labels can be elements of any well-ordered set $\mathcal{L}$.

\vspace{0.5cm}

\textsc{Labeling by $\mathbb{Z}$}. In Section~\ref{sec:Z} we consider countable graphs with vertices and edges labeled by $\mathbb{Z}$.  This new type of labeling allows us to consider paths that continue infinitely in both directions.

\begin{definition}\label{2sidedpath}
Given a countable graph $G=(V,E)$ and a labeling $\phi$ of $V$ by $\mathbb{Z}$, an \emph{infinite increasing two-sided path} in $G$ is an infinite path $\{v_i: -\infty < i < \infty\}$, with $\phi(v_i) < \phi(v_{i+1})$ for all $i$.
\end{definition}
For the existence of an infinite increasing two-sided path we find the following equivalent condition.
\begin{customthm}{\ref{thm:graph_v_z_2}}
For a countable graph $G=(V,E)$ the following are equivalent.
\begin{enumerate}
    \item[$(1)$] For every partition of $V$ into infinite sets $V_1, V_2$, there exist $W_1\subseteq V_1$, $W_2 \subseteq V_2$ so that for $i=1,2$ and for all $v\in W_i$, $v$ has infinite degree in $W_i$, and so that there exists an edge $w_1 w_2 \in E$ with $w_1\in W_1$ and $w_2 \in W_2$.
    \item[$(2)$] For any labeling of $V$ with $\mathbb{Z}$ there exists an infinite increasing two-sided path.
\end{enumerate}
\end{customthm}
Perhaps somewhat unexpectedly, for edge labeling the following holds.
\begin{customprop}{\ref{edgeZ1sided}}
For a countable graph $G=(V,E)$ there exists a labeling of $E$ by $\mathbb{Z}$ such that there is no infinite increasing path in $G$.
\end{customprop}

\textsc{Directed Graphs}.  In Section~\ref{sec:directed} we consider directed graphs whose vertices or edges are labeled by $\mathbb{N}$ or $\mathbb{Z}$.  These theorems correspond to theorems about undirected graphs given previously in this paper. Let $D=(V,E)$ be a directed graph where an edge $(u,v)\in E$ is oriented from $u$ to $v$.

\begin{customprop}{\ref{divertN}}
For a directed graph $D=(V,E)$, the following are equivalent.
\begin{enumerate}
\item[$(1)$] There exists an induced subgraph $D'$ of $D$ of which all vertices have infinite out-degree in $D'$.
\item[$(2)$] For any labeling of $V$ by $\mathbb{N}$, there exists an infinite increasing directed path in $D$.
\item[$(3)$] For any labeling of $E$ by $\mathbb{N}$, there exists an infinite increasing directed path in $D$.
\end{enumerate}
\end{customprop}

\begin{customprop}{\ref{thm:oriented_v_z_1}}
For a directed graph $D=(V,E)$, the following are equivalent.
\begin{enumerate}
    \item[$(1)$] For every partition of $V$ into infinite sets $V_1, V_2$, there exists $W_1 \subseteq V_1$, $W_2 \subseteq V_2$ so that for $i=1,2$ and for all $v\in W_i$, $v$ has infinite out-degree in $D[W_i]$.
    \item[$(2)$] For any labeling of $V$ by $\mathbb{Z}$ there is an infinite increasing directed path.
\end{enumerate}
\end{customprop}

\begin{definition}\label{def:inf2dipath}
Given a directed graph $D=(V,E)$ and a labeling $\phi$ of $V$ by $\mathbb{Z}$, we say $D$ contains an \emph{infinite two-sided directed path} if there exists an infinite directed path $\{v_i\}_{-\infty}^\infty$ in $D$ with $\phi(v_i) < \phi(v_{i+1})$ for all $i$.
\end{definition}

\begin{customprop}{\ref{thm:oriented_v_z_2}}
For a directed graph $D=(V,E)$, the following are equivalent.
\begin{itemize}
    \item For any partition of $V$ into infinite sets $V_1, V_2$ there exists $W_1 \subseteq V_1$, $W_2 \subseteq V_2$ so that for all $v\in W_1$, $v$ has infinite out-degree in $D[W_1]$, and for all $v\in W_2$, $v$ has infinite in-degree in $D[W_2]$, and there are vertices $w_i \in W_i$ for $i=1,2$, such that $(w_2,w_1) \in E$.
    \item For any labeling of $V$ by $\mathbb{Z}$ there is an infinite two-sided directed path.
\end{itemize}
\end{customprop}

\begin{customprop}{\ref{oriented_edgeZ1sided}}
For every countable directed graph $D=(V,E)$ there exists a labeling $\phi$ of $E$ by $\mathbb{Z}$, such that there is no infinite increasing directed path in $G$.
\end{customprop}


\section{Countable hypergraphs}\label{sec:hyper}

\textsc{Labeling vertices.}

We start this section by proving the following theorem. Recall Definitions~\ref{def:infincloosepath},~\ref{def:C_l}, and~\ref{def:infbranchingtree}.

\begin{theorem}\label{infbranchingtree}
For a countable, $k$-uniform hypergraph $H= (V,E)$ the following are equivalent.
\begin{enumerate}
    \item[$(1)$] $H$ has property $C_k$.
    \item[$(2)$] For any labeling of $V$ with $\mathbb{N}$ there exists an infinite increasing $(k-1)$-branching tree.
    \item[$(3)$] For any labeling of $V$ with $\mathbb{N}$ there exists an infinite increasing path.
\end{enumerate}
\end{theorem}

In the proof of this theorem we use Theorem~\ref{RM}.
\begin{proof}[Proof of Theorem~\ref{infbranchingtree}]
\textbf{ {(1)} $\Rightarrow$ {(2)}}.
Let $V'$ be the subset of $V$ with the property described by \emph{(1)}.
Then for any labeling of $V$ with $\mathbb{N}$ any increasing $(k-1)$-branching tree of depth $\ell$ can be extended to an increasing branching tree of depth $\ell+1$ using vertices of $V'$. Consequently there exists an infinite increasing $(k-1)$-branching tree.

\textbf{ {(2)} $\Rightarrow$ {(3)}}.
This is obvious, since an infinite increasing path can be found in any infinite increasing $(k-1)$-branching tree along some branch of the tree.

\textbf{ {(3)} $\Rightarrow$ {(1)}}.
In a hope to derive a contradiction suppose (3) holds but \emph{(1)} does not. Then there is a well-ordering $\prec$ of $V$ so that for every $v\in V$, the set
$$\{ e\in E: v\in e \text{ and } \forall u\in e \text{ with } u\neq v, v \prec u\}$$
 is finite.  For any edge $e\in E$ let $v_e$ be the minimal vertex of $e$ (i.e. such that $v_e \prec u$ for all $u\in e$, $u\neq v_e$). Let
$$ S_e = \Big\{ \{v_e, u\}: u\in e, u\neq v_e \Big\}$$
be a (2-uniform) edge set of a star with center in $v_e$. Consider a graph $G$ defined by
$$V(G) = V(H),$$
$$E(G) = \bigcup_{e\in E(H)} S_e.$$

Observe that for any $v\in V(G)$, the set
$$\Big\{u: \{v,u\}\in E(G), v\prec u\Big\}$$ remains finite. Therefore there is no subset $V' \subseteq V(G)$ with $\deg_{G[V']}(v) =\infty$ for every $v\in V'$.
Consequently by Theorem~\ref{RM}, there is a vertex-ordering $\phi$ of $V(G)$ by $\mathbb{N}$ so that $G$ contains no infinite increasing path.

For contradiction assume that $P = e_0,\ldots, e_i, \ldots$ is an infinite increasing path in $H$ with respect to $\phi$, and let $V(P) = \{v_1, v_2, \ldots\}$ where $\phi(v_i) < \phi(v_{i+1})$ for $i\geq 1$. In particular $e_i=\{v_{(k-1)i+1}, v_{(k-1)i+2}, \ldots, v_{(k-1)i+k}\}$ for $i\geq 0$. For each $i$ let $u_i = v_{(k-1)i+1}$ and $w_i = v_{(k-1)i+k}$ be the minimal and maximal vertex of $e_i$ with respect to $\phi$. Note that $u_i$ is not necessarily equal to $v_{e_i}$, which is the minimal vertex of $e_i$ with respect to the well-ordering $\prec$.

In order to arrive at a contradiction we construct an infinite increasing path $P'$ in $G$ as a union of edges and 2-paths as follows. For each $i=0,1, \ldots$ define a path $P_i$ to be
\begin{itemize}
\item $u_iw_i$ if $v_{e_i} \in \{u_i, w_i\}$.
\item $u_iv_{e_i}w_i$ if $v_{e_i} \not\in \{u_i, w_i\}$.
\end{itemize} 
Finally, set $P'$ to be a path that is the concatenation of $P_i$-s with increasing labels on vertices, which contradicts our assumption of not (1).
\end{proof}

For a $k$-uniform hypergraph, property $C_\ell$ clearly implies property $C_{\ell-1}$, for $\ell\leq k$.  In order to see that property $C_{\ell-1}$ does not imply property $C_\ell$, consider a $k$-uniform hypergraph $H$ obtained from an infinite complete $(\ell-1)$-uniform hypergraph, the edges of which are extended by pairwise disjoint $(k-\ell+1)$-tuples. $H$ clearly has property $C_{\ell-1}$ but not property $C_\ell$.

However for property $C_2$ we still get a result analogous to Theorem~\ref{infbranchingtree}. 
\begin{definition}
Let $G=(V,E)$ be a $k$-uniform hypergraph. For a labeling $\phi$ of $V$ by $\mathbb{N}$ we say that infinite path $e_0, \ldots, e_i,\ldots$ is \emph{skip-increasing} if there exist vertices~$v_0, v_1, v_2, \ldots$ so that for every $i\geq 0$, $\phi(v_i) < \phi(v_{i+1})$ and  $\{v_i, v_{i+1}\} \subset e_i$.

\end{definition}

By mimicking the proof of Theorem~\ref{infbranchingtree} one can get the following:

\begin{proposition}
For a countable, $k$-uniform hypergraph $H =(V,E)$ the following are equivalent.
\begin{enumerate}
\item[$(1)$] $H$ has property $C_2$.
\item[$(2)$] For any labeling of $V$ with $\mathbb{N}$ there exists an infinite skip-increasing path in $G$.
\end{enumerate}\end{proposition}

\textsc{Labeling edges.} Now, we consider edge-labelings of hypergraphs. Unfortunately we were not able to find sufficient and necessary conditions. Recall Definitions~\ref{def:infincloosepathE} and~\ref{def:infbranchingtreeE}.

\begin{theorem}~\label{thm:hyperedge} Let $H= (V,E)$ be a simple, countable, $k$-uniform hypergraph. Consider the following conditions:
\begin{enumerate}
    \item[$(1)$] For any labeling $\phi$ of $E$ by $\mathbb{N}$ there is an infinite increasing $(k-1)$-branching tree in $H$.
    \item[$(2)$] $H$ has property $C_2$.
    \item[$(3)$] For any labeling $\phi$ of $E$ by $\mathbb{N}$ there is an infinite increasing path in $H$.
\end{enumerate}
Then $(1) \Rightarrow (2) \Rightarrow (3)$.
\end{theorem}
\begin{proof}
We start with showing the easier implication ${(2)}\Rightarrow {(3)}$. Let $H$ have property $C_2$, that is, there is a set $V' \subseteq V$ such that $\forall v \in V'$, $|\{e\in E: v\in e\; \text{and} \; |V'\cap e|\geq 2\}| = \infty$.
Let $\phi$ be any labeling of $E$ by $\mathbb{N}$. Set $e_0=\emptyset$ and let $e_1\in E$ be such that $|V'\cap e|\geq 2$.
Assume we have constructed an increasing path $e_1, e_2, \ldots, e_j$ and assume there exists some vertex $v \in e_j\setminus e_{j-1}$ that also belongs to $V'$. Since $H$ is simple and $v$ has infinite degree, there are infinitely many edges $e$ containing $v$ with $e_i \cap (e \setminus \{v\})= \emptyset$ for all $i=1, \ldots, j$ and with $\phi(e)>\phi(e_j)$.
Choose any one of these edges to be $e_{j+1}$, and observe that there exists some vertex $u \in e_{j+1} \setminus e_{j}$ with $u\in V'$.
In this inductive way we form the infinite increasing path $e_1, e_2, e_3, \ldots$.

Now, we show {(1)} $\Rightarrow$ {(2)}. We will show that if $H$ does not satisfy {(2)} it does not satisfy {(1)} as well.
Due to our assumption of not (2), there exists a well-ordering $\prec$ of $V$ so that for every $v\in V$, the set
$$\{\{v, u_1, \ldots, u_{k-1} \} \in E: v \prec u_i \text{ for some } i \}$$
is finite.  Consequently we have
\begin{fact}\label{factstar}
For any vertex $v\in V$, there are finitely many edges in  $E$ in which $v$ is not the maximal vertex with respect to $\prec$.
\end{fact}

Consider also an arbitrary labeling of $V$ by $\mathbb{N}$, where order is denoted by $<$.
If $e = \{v_1, v_2, \ldots, v_k \}$ with $v_1 < v_2 < \cdots < v_k$, we say $\ell(e)=v_1$ and $s(e)=v_2$, for the least and second-least vertices.
Divide the edges of $E$ into Type I edges ($E_I$) and Type II edges ($E_{II}$) so that
$$E_I = \{e: \ell(e) \text{ is \emph{not} the maximum vertex of } e \text{ with respect to } \prec\}$$
$$E_{II} = E \setminus E_I.$$
Due to Fact~\ref{factstar} and the fact that every natural number has only finitely many predecessors, we infer that
\begin{proposition}\label{propType1}
Any vertex $v\in V$ can be in only finitely many Type I edges.
\end{proposition}

We will construct separate labelings of $E_I$ and $E_{II}$, each of which forbids an infinite increasing $(k-1)$-branching tree.
Let $H_I = (V,  E_I)$.
Note that each vertex of $H_I$ has finite degree.

\begin{claim}\label{claim3} 
There is a labeling $\phi$ by $\mathbb{N}$ of $H_I$ with no infinite increasing path.
\end{claim}
\begin{proof}[Proof of Claim]
The edge set  $E_I$ can be partitioned into finite sets  $E_1, E_2, \ldots$ so that for all edges $e \in E_i$, any edge $f\in E_I$ incident to $e$ is in  $E_j$, $j\leq i+1$. (E.g. we may set $E_i$ to be the edges that are distance $i-1$ away from some fixed edge $e$.)
Give to  $E_I$ a labeling $\phi$ by $\mathbb{N}$ so that for all $i\in \mathbb{N}$,
$$\phi(e) < \phi(f) \text{ for all } e\in  E_{2i-1}, f\in  E_{2i},$$
$$\phi(f) > \phi(e) \text{ for all } f\in  E_{2i}, e\in  E_{2i+1}.$$
If an increasing path in $E_I$ uses an edge from any  $E_{2i-1}$ or $E_{2i}$, then the path cannot later use any edges from $E_j$ for any $j\geq 2i+1$. Such a path must be finite, so $H_I$ can contain no infinite increasing path.
\end{proof}

Let $H_{II} = (V,  E_{II})$.
We now construct a labeling $\psi$ by $\mathbb{N}$ on $E_{II}$ in the following way.  Suppose $e_1, e_2\in H_{II}$.
\begin{enumerate}
\item If $s(e_1) < s(e_2)$, then $\psi(e_1) <  \psi(e_2)$.
\item If $s(e_2) = s(e_1)$ and $\ell(e_1) \prec \ell(e_2)$, then $\psi(e_1) > \psi(e_2)$.
\end{enumerate}
Recall that we want to show that $H_{II}$ does not contain an infinite increasing $(k-1)$-branching tree with respect to $\psi$. Assume the contrary, that is that $H_{II}$ contains such a tree $T$. We now recursively construct a branch of $T$ that is a path $e_1, e_2, e_3, \ldots $ (with $t(e_i)=i$) satisfying $\ell(e_1) \succeq \ell(e_2) \succeq \ell(e_3) \succeq \cdots$.
Assume we have constructed a path $e_1, \ldots, e_r$ with $\ell(e_1) \succeq \ell(e_2)\succeq \cdots \succeq \ell(e_r)$ for some $r\geq 1$.
We choose to extend the path through either $\ell(e_r)$ or $s(e_r)$.
At least one of these vertices is not in $e_{r-1}$ (or not a root in case $r=1$), and that vertex is incident to an edge $e_{r+1}\in E(T)$ with $\psi(e_{r+1}) > \psi(e_r)$.
Recall that $\ell(e)$ is the least and $s(e)$ is the second-least vertex in $e$ with respect to $<$. We also note that $\ell(e)\succ v$ for all $v\in e$, $v\neq \ell(e)$ for any Type II edge $e$. Let $v = e_r \cap e_{r+1}$ and consider the following exhaustive cases, of which only three can happen.
\begin{enumerate}
\item If $v = \ell(e_r) = \ell(e_{r+1})$, then of course $\ell(e_r) \succeq \ell(e_{r+1})$.
\item It is impossible that $v = \ell(e_r)$ and $v \neq \ell(e_{r+1})$. If so then $v \geq s(e_{r+1})> \ell(e_{r+1})$, implying $s(e_{r+1}) \leq v= \ell(e_r) < s(e_r)$. This contradicts that $\psi(e_{r+1}) > \psi(e_r)$.
\item If $v = s(e_r)$ and $v = \ell(e_{r+1})$, and since $\ell(e_r) \succ s(e_r)$ for all Type II edges, \\$\ell(e_r) \succ v=\ell(e_{r+1})$.
\item If $v = s(e_r) = s(e_{r+1})$, then since $\psi(e_{r+1}) > \psi(e_{r})$, by the definition of $\psi$, it must be that $\ell(e_{r}) \succ \ell(e_{r+1})$.
\item It is impossible that $v = s(e_r)$ and $v > s(e_{r+1})$. If so then $s(e_{r+1}) < s(e_r)$, which contradicts to $\psi(e_{r+1}) > \psi(e_r)$.
\end{enumerate}
So for all possible ways in which $e_{r+1}$ intersects $e_r$, we have $\ell(e_r) \succeq \ell(e_{r+1})$, and so $\ell(e_1) \succeq \ell(e_2) \succeq \ell(e_3) \succeq \cdots$.
There are no consecutive equalities (otherwise three edges on a loose path would intersect at a single vertex).
But this is a contradiction, since $\succ$ is a well-ordering and so every decreasing sequence must be finite. Hence, $H_{II}$ has no infinite increasing $(k-1)$-branching tree with respect to $\psi$.

Having shown that for both $H_I$ and $H_{II}$ there exist edge-orderings $\phi$ and $\psi$ by $\mathbb{N}$ that forbid any infinite increasing $(k-1)$-branching trees, we will construct an edge-ordering $\gamma$ of $H$ forbidding any infinite increasing $(k-1)$-branching trees.

Let $\gamma$ be a labeling of $E$ by $\mathbb{N}$ so that
$$\text{ for all } e,f\in  E_{I}, \gamma(e) < \gamma(f) \text{ iff } \phi(e) < \phi(f),$$
$$\text{ for all } e,f\in  E_{II}, \gamma(e) < \gamma(f) \text{ iff } \psi(e) < \psi(f), \text{ and}$$
$$ \text{ for all } e\in  E_{I}, f\in  E_{II} \text{ with } e \text{ incident to } f, \gamma(e) < \gamma(f),$$
which is possible because each vertex of $e$ has finite degree in $H_I$.

The labeling $\gamma$ inherits the labeling of $E_{I}$ by $\phi$ and the labeling of  $E_{II}$ by $\psi$.
Neither of these edge sets contains an infinite increasing $(k-1)$-branching tree with their respective labeling.
Suppose $H$ contains an infinity increasing $(k-1)$-branching tree $T$ by $\gamma$.  At least one edge $e$ of $E_{II}$ must be used, and then every edge following $e$ in $T$ must also be from  $E_{II}$, by the definition of $\gamma$.
This would require that  $E_{II}$ contains an infinite increasing $(k-1)$-branching tree, a contradiction.  So there exists a labeling of $E$ by $\mathbb{N}$ containing no infinite increasing $(k-1)$-branching tree.

\end{proof} 

We also notice that in the theorem above ${(2)}\not \Rightarrow {(1)}$. Consider a complete countable graph $G$ with vertices $V=\{v_1,v_2, \ldots\}$ and complement every edge $e=\{v_i,v_j\}$ by $k-2$ new vertices $u_{i,j,1}, u_{i,j,2}, \ldots, u_{i,j,k-2}$ to form a $k$-uniform hypergraph $H=(V\cup U,E)$ with $E=\{\{v_i,v_j,u_{i,j,1}, u_{i,j,2}, \ldots, u_{i,j,k-2}\}: i,j\in \mathbb{N}\}$. The original vertex set $V$ satisfies \emph{(2)}. On the other hand, there is no infinite $(k-1)$-branching tree in $H$ because every hyperedge contains $k-2$ vertices of degree 1.       

It remains unknown to us if {(2)} is equivalent to {(3)} or not, even for the case $k=3$.


\section{Property FIN}\label{sec:FIN}
We first prove the following claim concerning $\chi^*$.  Recall the definitions of property FIN (Definition~\ref{def:fin}) and $\chi^\star(G)$ (Definition~\ref{def:chistar}).

\begin{definition}\label{def:fin_k}
We say that for an integer $k$ a graph $G$ has property FIN$_k$ if for any labeling of vertices of $G$ with $\mathbb{N}$ there exists an increasing path with $k$ vertices, but there exists a labeling with no increasing path of length $k+1$.  If no such $k$ exists, then $G$ has property FIN.
\end{definition}

The following is a simple modification of the well-known Gallai-Hasse-Roy-Vitaver Theorem \cite{Gal:68}, \cite{Has:65}, \cite{Roy:67}, \cite{Vit:62}.

\begin{claim}\label{cl:fin_k}
A graph $G$ has property FIN$_k$ if and only if $\chi^*(G) = k$.
\end{claim} 
\begin{proof}
Let $G = (V,E)$ have property FIN$_k$ and consider a labeling of $V$ by $\mathbb{N}$ so that $G$ does not contain any increasing path of length larger than $k$.
Let $V_k$ be the set of vertices that are maximal (i.e. are not adjacent to any vertex of larger label).
Note that $V_k$ cannot be empty, or else $G$ would contain an increasing path of infinite length.
Delete $V_k$ from $V$ to obtain the set $U_{k-1} = V - V_k$ and consider $G_{k-1}$, the subgraph of $G$ induced on $U_{k-1}$.
Let $V_{k-1}$ be the set of all vertices that are maximal within $G_{k-1}$.  We delete $V_{k-1}$ from $U_{k-1}$ to obtain $U_{k-2}$ and let $G_{k-2}$ be induced on $U_{k-2}$.
We continue this way until we exhaust all vertices
of $G$. Note that for all $j\in[k-1]$ and every vertex $v\in V_j$, there is a vertex $u\in V_{j+1}$ adjacent to $v$ and with larger label (otherwise, $v$ itself should belong to $V_{j+1},$ not $V_j$).
Now we show that $V_0$ is empty. If there were some vertex $v_0$ in $V_0$, we would have an increasing path of length $k+1$ beginning at $v_0$.
Consequently $V_1, V_2, \ldots, V_k$ constitutes a partition of $V$.

Further, observe that each set $V_i$, $1\leq i\leq k$, is an independent set, since two adjacent vertices in some graph cannot both be maximal.
We say that vertex $w$ dominates $v$ if $\{w,v\}$ is an edge and label of $w$ is larger than of $v$. Since vertices are labeled by integers, each vertex $v\in V_i$ can dominate only finitely many vertices of $V_{1}, V_{2},\ldots ,V_{i-1}$ and cannot be dominated by any vertices in those sets, so $v$ has finite degree into $V_{1}, V_{2},\ldots ,V_{i-1}$. So, if $G$ has property FIN$_k$, then $\chi^*(G) \leq k$.

On the other hand if $\chi^*(G) =k$ with partition $V_1, \ldots, V_k$, we can label vertices of $V_1$ arbitrarily and then label vertices of $V_{2},  \ldots, V_k$ in succession, always making sure that the label of any vertex $v\in V_j$ is higher than that of its neighbors from $V_{1},\ldots, V_{j-1}$.
This is possible because $v$ has finite degree into $V_{1},\ldots, V_{j-1}$ by $\chi^*(G)$.
This labeling prevents an increasing path of length $k+1$.
Therefore $G$ has property FIN$_j$ for some $j\leq k$. 

So, if $\chi^*(G) =k$, then we have that $G$ has property FIN$_j$ for some $j\leq k$ and then $k=\chi^*(G) \leq j\leq k$, which implies that $G$ has property FIN$_k$. Similarly, if $G$ has property FIN$_k$, then $\chi^*(G)=k$.
\end{proof}

Now we are ready to prove Theorem~\ref{thm:FIN}.\\
\begin{theorem}\label{thm:FIN}
A graph $G$ has property FIN if and only if $\chi^*(G)=\infty$ and $G$ does not have a subgraph in which all degrees are infinite.
\end{theorem}
\begin{proof}
First, suppose $G$ has property FIN. Since for every labeling of $V(G)$ by $\mathbb{N}$ there is an increasing path of arbitrary finite length, $G$ does not have property FIN$_k$ for any $k$, and hence by Claim~\ref{cl:fin_k}, $\chi^\star(G) = \infty$. Also, $G$ cannot contain a subgraph in which all degrees are infinite, for otherwise Theorem~\ref{RM} implies that every labeling of $V(G)$ by $\mathbb{N}$ would contain an infinite increasing path, which would contradict FIN.

Now suppose $G$ does not have property FIN.  Then either there exists a labeling of $V(G)$ by $\nn$ forbidding increasing paths of length $k$ for some $k$, or for every labeling of $V(G)$ there is an infinite increasing path.  In the former case, $G$ has property FIN$_j$ for some $j<k$, implying by Claim~\ref{cl:fin_k} that $\chi^\star(G)=j < k$.  In the latter case, Theorem~\ref{RM} implies $G$ has a subgraph in which all degrees are infinite.
\end{proof}
Clearly, if there is a labeling of $V(G)$ by $\mathbb{N}$ that prevents an increasing path of $k+1$ vertices, then $\chi(G)\leq k$. Unfortunately, the converse is not true. 

\begin{proposition}\label{chistarnotchi}
There exists a bipartite graph $H$ that has property FIN.
\end{proposition}

Hence, we cannot replace condition $\chi^*(G)=\infty$ with $\chi(G)=\infty$ in the statement of Theorem~\ref{thm:FIN}. To prove Proposition~\ref{chistarnotchi}, we first define a ``half-graph" and then construct graph $H$.

\textbf{Definition of a half-graph}

Start with $I$ and $F$ -- two copies of $\mathbb{N}$. 
Let $G=(V,E)$ be a graph  with  $V(G)=I\cup F$ and $$E=\{\{x,y\} \; : \; x \in I, \; y \in F, \; x\leq y\}.$$ Graph $G$ defined on $I\cup F$ in such a way is called the half-graph and is denoted by $G[I,F]$. Note that all vertices of $I$ have infinite degree in $G$, while vertices of $F$ have finite degree. See Figure~\ref{G_graph} for an illustration.
 
\begin{figure}[h]
\begin{minipage}{\linewidth}
      \centering
      \begin{minipage}{0.41\linewidth}
          \begin{figure}[H]
              \includegraphics[width=\linewidth]{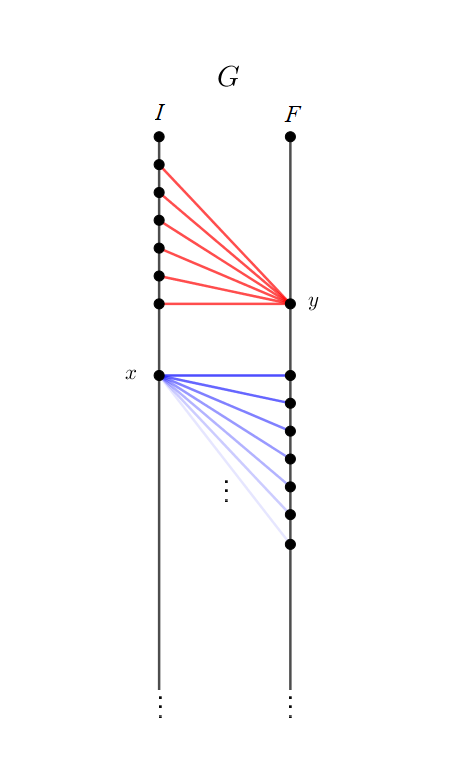}
              \caption{Each $x\in I$ has infinite degree, each $y \in F$ has finite degree. \label{G_graph}}
          \end{figure}
      \end{minipage}
      \hspace{0.01\linewidth}
      \begin{minipage}{0.55\linewidth}
          \begin{figure}[H]
              \includegraphics[width=\linewidth]{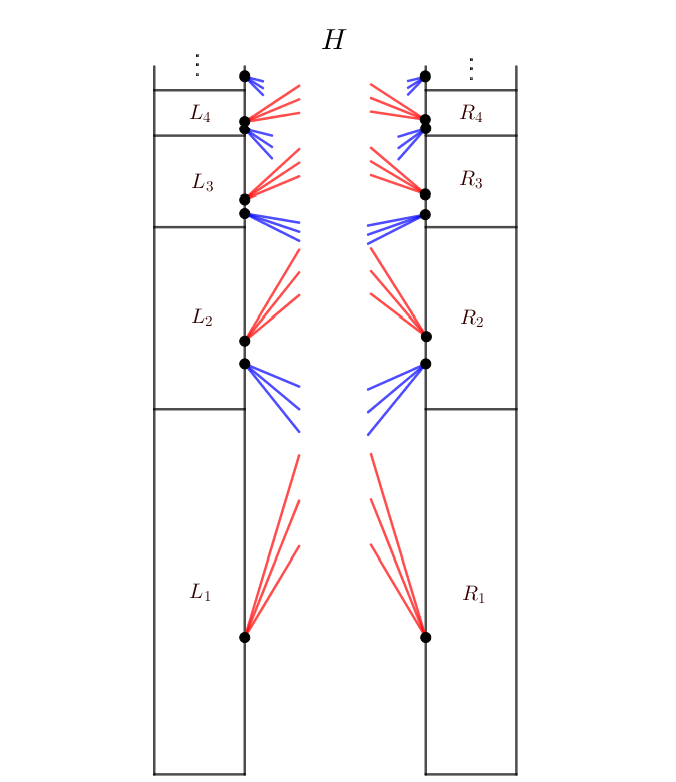}
              \caption{Each $x\in V_i^L$ has finite degree to $U_i^R$ and infinite degree to $V_j^R$ for $1\leq j<i$. \label{H_graph}}
          \end{figure}
      \end{minipage}
  \end{minipage}
\end{figure}
 
\textbf{Construction of $H$.}

 Let $L$ and $R$ be again two copies of $\mathbb{N}$. Then set $V(H)=L \cup R$. 
To define $E(H)$ we need some preliminary definitions. For each $i\geq 1$, 
$$L_i=\{n \in L \; :  n = 2^{i-1}(2j-1) , \;  j \in \nn\}, $$
$$R_i=\{n \in R \; :  n = 2^{i-1}(2j-1) , \;  j \in \nn\}, $$

Note, that from this definition we have that $L=\bigcup_{i=1}^\infty L_i$ and $R=\bigcup_{i=1}^\infty R_i$. Notice that we can view each $L_i$ (or $R_i$) as a copy of $\mathbb{N}$ with the order inherited from $L$ (or $R$).

We define a graph $H_1$ on $V(H)$ where, for each $i$, there is a copy of $G$ between sets $L_i$ and $R_j$ for each $j>i$ ($L_i$ is the set with finite degrees in $G$). More precisely, 
$$E(H_1)=\bigcup_{i=1}^\infty \left( \bigcup_{j=i+1}^\infty G[R_j, L_i] \right).$$
Each vertex in $L$ has finite degree in $H_1$, and each vertex in $R$ has infinite degree in $H_1$.
Similarly we define a graph $H_2$ on $V(H)$ where, for each $i$, there is a copy of $G$ between sets $R_i$ and $L_j$ for each $j>i$ ($R_i$ is the set with finite degrees in $G$). More precisely, 
$$E(H_2)=\bigcup_{i=1}^\infty \left( \bigcup_{j=i+1}^\infty G[L_j, R_i] \right).$$
Each vertex in $R$ has finite degree in $H_2$, and each vertex in $L$ has infinite degree in $H_2$. Also note that $E(H_1)\cap E(H_2)=\emptyset$.
Then $$E(H) = E(H_1) \cup E(H_2).$$

Notice that, by a construction, $H$ is bipartite. See Figure~\ref{H_graph} for an illustration.
\begin{proof}[Proof of Proposition~\ref{chistarnotchi}]
\begin{claim}
$H$ does not contain a subgraph with infinite degrees.
\end{claim}
\begin{proof}
Assume that $H$ does contain a subgraph $H^\prime$ with all degrees infinite. Let $x \in L_t$ be a vertex of $H^\prime$ with the smallest possible $t$.

Notice that $x$ has a finite degree into each $\bigcup_{i=t+1}^\infty R_i$, since there are only finitely many vertices in $R$ less than or equal to $x$. Also, $x$ has no neighbours in $R_t$. So in order for $x$ to have infinite degree in $H^\prime$, $H^\prime$ has to contain infinitely many vertices of $\bigcup_{i=1}^{t-1}R_i$. Let $y$ be the smallest vertex in $V(H^\prime)\cap \bigcup_{i=1}^{t-1}R_i$ such that $xy \in E(H^\prime)$. Assume that $y \in R_s$ for some $s<t$. Now, $y$ has finite degree in $\bigcup_{i=s+1}^\infty L_i$ and none of the vertices of $H^\prime$ are in $\bigcup_{i=1}^{s-1}L_i$ (by minimality of $t$). Hence, $y$ has a finite degree in $H^\prime$, a contradiction.
\end{proof}
Hence by Theorem~\ref{RM} there is a labeling of $V(H)$ by $\nn$ with no infinite increasing path.
\begin{claim}
Any labeling of $H$ by $\mathbb{N}$ contains an increasing path of arbitrary length.
\end{claim}
\begin{proof}
Notice that any vertex $v$ of $L_{t+1}$ has infinite neighborhood in $R_t$ for any $t\geq 2$. 

Let $V(H)$ be labeled by $\mathbb{N}$ and let $k$ be an integer. Then for any $v_1 \in L_{k+1}$ there is $v_2 \in R_{k}$ adjacent to $v_1$ and with larger label. Similarly for $v_2 \in R_{k}$ there is $v_3 \in L_{k-1}$ adjacent to $v_2$ and with larger label. Continuing this way we obtain an increasing path of length $k$. 
\end{proof}
Therefore, $H$ has property FIN.
\end{proof}


\section{Labeling by $\mathbb{Z}$}\label{sec:Z}
Now we consider labelings using all the integers instead of just the natural numbers. Given a countable graph $G=(V,E)$ and a labeling $\phi$ of $V$ by $\mathbb{Z}$, we say $G$ contains an \emph{infinite increasing path} if there exists an infinite path of vertices $\{v_i\}_{i=1}^\infty$ in $G$ such that $\phi(x_i) < \phi(x_{i+1})$ for all $i\geq 1$.

\begin{proposition}\label{thm:graph_v_z_1}
For a countable graph $G=(V,E)$ the following are equivalent
\begin{enumerate}
    \item[$(1)$] For every subset $V_1 \subseteq V$ where both $V_1$ and $V\setminus V_1$ are infinite, there exist $W_1\subseteq V_1$ such that all $v\in W_1$ have infinite degrees in $W_1$.
      \item[$(2)$] For any labeling of $V$ with $\mathbb{Z}$ there exists an infinite increasing path.
\end{enumerate}
\end{proposition}

\begin{proof}
One implication is clear.
Suppose $(1)$ and let $\phi$ be any labeling of $V$ by $\mathbb{Z}$.
We let
$$V_1 = \{v\in V: \phi(v) > 0\}$$
$$V_2 = \{v\in V: \phi(v) \leq 0\}.$$
By $(1)$ there exists a set $W_1\subseteq V_1$ so that for all $v\in W_1$, $v$ has infinite degree in $G[W_1]$.
Then by Theorem~\ref{RM}, there exists an infinite increasing path in $G[W_1]$ and hence in $G$.

Now suppose $(1)$ does not hold.
Then there exists a partition of $V$ into infinite sets $V_1, V_2$ so every subset $W_1\subseteq V_1$ has a vertex of finite degree in $G[W_1]$.
Then by Theorem~\ref{RM}, there exists a vertex-labeling of $G[V_1]$ by $\nn$ containing no infinite increasing path.
Label the vertices of $V_2$ arbitrarily by $\mathbb{Z}\setminus \nn$. Since all vertices with positive label are contained in $G[V_1]$, and since $G[V_1]$ contains no infinite increasing path, neither does $G$.
\end{proof}

Labeling using $\mathbb{Z}$ instead of $\nn$ allows us to consider paths that continue infinitely in both directions. Recall Definition~\ref{2sidedpath}.

\begin{theorem}\label{thm:graph_v_z_2}
For a countable graph $G=(V,E)$ the following are equivalent
\begin{enumerate}
    \item[$(1)$] For every partition of $V$ into infinite sets $V_1, V_2$, there exist $W_1\subseteq V_1$, $W_2 \subseteq V_2$ so that for $i=1,2$ and for all $v\in W_i$, $v$ has infinite degree in $G[W_i]$, and so that there exists an edge $w_1 w_2 \in E$ with $w_1\in W_1$ and $w_2 \in W_2$.
    \item[$(2)$] For any labeling of $V$ with $\mathbb{Z}$ there exists an infinite increasing two-sided path.
\end{enumerate}
\end{theorem}
\begin{proof}

One implication is clear.
Suppose {(1)} and let $\phi$ be any labeling of $V$ by $\mathbb{Z}$.
We let
$$V_1 = \{v\in V: \phi(v) \geq 0\},$$
$$V_2 = \{v\in V: \phi(v) < 0\}.$$
Let $W_1$ and $W_2$ be the subsets of $V_1$ and $V_2$ respectively that are ensured by {(1)}, and let $w_1\in W_1$ and $w_2 \in W_2$ be the adjacent vertices.
Let $v_1 = w_1$, and for $i\geq 1$, let $v_{i+1}$ be a neighbor of $v_{i}$ in $W_1$ so that $\phi(v_{i+1}) > \phi(v_{i})$.
Since each $v_{i}$ has infinitely many such neighbors, $\{v_i\}_{i=1}^\infty$ forms an infinite increasing path in $W_1$.
Similarly let $v_0 = w_2$ and for $i\leq 0$, let $v_{i-1}$ be a neighbor of $v_{i}$ in $W_2$ so that $\phi(v_{i-1}) < \phi(v_{i})$.
Since $v_1$ and $v_0$ are connected by an edge in $G$, the path formed by $\{v_i: -\infty < i < \infty\}$ is an infinite increasing two-sided path in $G$.

Now suppose $(1)$ does not hold. Then there are two possible cases.
\\ {\bf Case I} - there exists a partition of $V$ into infinite sets $V_1, V_2$ such that for $i=1$ or $i=2$ no subset $W_i \subseteq V_i$ induces a subgraph with all infinite degrees.
Then by Proposition~\ref{thm:graph_v_z_1}, there is a labeling of $V$ by $\mathbb{Z}$ with no infinite increasing path.  Such a labeling clearly forbids an infinite increasing two-sided path.

{\bf Case II} - there exists a partition of $V$ into infinite sets $V_1, V_2$ so that for any subsets $W_1\subseteq V_1$, $W_2 \subseteq V_2$  with all vertices of $W_i$ have infinite degree in $W_i$, there is no edge between vertices of $W_1$ and $W_2$.
For $i=1,2$, consider the family $\mathcal{W}_i$ of such sets $W_i$, where $\mathcal{W}_i$ is partially ordered by inclusion, and observe that each $\mathcal{W}_i$ contains a maximal element. Fix vertex-maximal subsets $W_1\in \mathcal{W}_i$ and $W_2\in \mathcal{W}_i$.
With $U_i = V_i\setminus W_i$ for $i=1,2$, the vertex-maximality of $W_i$ implies
\begin{enumerate}
\item[(i)] every nonempty subset $U \subseteq U_i$ contains a vertex $v\in U$ with finite degree in $G[U]$,
\item[(ii)] every $v\in U_i$ has finite neighborhood in $W_i$.
\end{enumerate}
By (i), Theorem~\ref{RM} implies that there is a labeling $\phi_U$ of $U_1$ by $\mathbb{N}$ that forbids infinite increasing paths in $U_1$.
By (ii) we can form a labeling $\phi$ of $V_1$ by positive integers that preserves $\phi_U$ (i.e. $\phi(u_1) < \phi(u_2)$ iff $\phi_U(u_1) < \phi_U(u_2)$ for all $u_1, u_2 \in U_1$) and where $\phi(u) > \phi(w)$ for all $u \in U_1$, $w \in W_1$  with $u$ adjacent to $w$.
Since $\phi$ preserves $\phi_U$, $U_1$ contains no infinite increasing path with respect to $\phi$.

Now follow a similar procedure for $V_2$ and $W_2$, with $U_2 = V_2 \setminus W_2$.
Here we label $V_2$ by \emph{negative} integers to prevent any infinite \emph{decreasing} path.
Construct a labeling $\psi$ of $V_2$ by negative integers so that for any $u\in U_2$, $w\in W_2$ with $u$ adjacent to $w$, $\psi(u) < \psi(w)$, and so that $U_2$ contains no infinite decreasing path with respect to $\psi$.

Finally, for all $v\in V(G)$, let
\[ \gamma(v) = 
	\begin{cases} 
      \phi(v) & v\in V_1 \\
      \psi(v)+1 & v\in V_2,
   \end{cases}
\]
where we add 1 to $\psi(v)$ so that 0 is used as a label.
Note that $\gamma$ is a labeling of $V(G)$ by $\mathbb{Z}$ and $\gamma(V_1)=\mathbb{N}, \gamma(V_2)=\mathbb{Z}\setminus \mathbb{N}$.
Observe the following with respect to $\gamma$:
\[ (3)
	\begin{cases} 
      \text{there are no infinite increasing paths in } U_1, \\
      \text{there are no infinite decreasing paths in } U_2, \\
      \text{there are no edges from } W_1 \text{ to } W_2. \\
   \end{cases}
\]
We claim that there is no infinite two-sided path in $G$ with respect to $\gamma$.
Suppose $\{v_i\}_{-\infty}^\infty$ is such a path, and without loss of generality assume $v_0\in V_2$, $v_1\in V_1$.
If $v_1\in U_1$, then all vertices $v_2, v_3, \ldots$ must be in $U_1$, and so $v_1, v_2, \ldots$ is an infinite increasing path in $U_1$, contradicting (3).
Similarly if $v_0\in U_2$, then $v_{-1}, v_{-2}, \ldots \in U_2$, so there is an infinite decreasing path in $U_2$, contradicting (3).
Consequently, $v_1\in W_1, v_0 \in W_2$, which contradicts (3).
Therefore $G$ contains no infinite increasing two-sided path.

\end{proof}

Perhaps interestingly, when labeling edges by $\mathbb{Z}$, the result is entirely different from the vertex-labeling case, as the next theorem shows.

\begin{proposition}\label{edgeZ1sided}
For every countable graph $G=(V,E)$ there exists a labeling $\phi$ of $E$ by $\mathbb{Z}$, such that there in no infinite increasing path in $G$.
\end{proposition}

\begin{proof}
Suppose $G$ contains a matching $M$ with infinitely many edges.
Define a labeling $\phi$ of $E$ by $\mathbb{Z}$ so that $\phi(e)>0$ if and only if $e\in M$.
Then there is no infinite increasing path in $G$, since no two edges with positive labels are incident.

Now suppose $G$ contains no infinite matching.
Consider a vertex-maximal matching with vertex set $V'$.
Then every edge of $E$ is incident to a vertex of $V'$.
Therefore the longest path in $G$ can have at most $2\vert V' \vert$ edges, so no infinite path exists, regardless of how the edges are labeled.
\end{proof}

Proposition~\ref{edgeZ1sided} immediately implies

\begin{corollary}\label{cor2}
For every countable graph $G=(V,E)$ there exists a labeling $\phi$ of $E$ by $\mathbb{Z}$, such that there is no infinite two-sided path in $G$.
\end{corollary}


\section{Directed Graphs}\label{sec:directed}

The next four propositions give results for directed graphs.  The proofs of these propositions mimic the proofs of Theorem~\ref{RM}, Proposition~\ref{thm:graph_v_z_1}, Theorem~\ref{thm:graph_v_z_2}, and Proposition~\ref{edgeZ1sided} respectively, and so they are not included in this paper\footnote{ For the sake of completeness, proofs of these results can be found in Appendix A.}. Let $D=(V,E)$ be a directed graph where an edge $(u,v)\in E$ is oriented from $u$ to $v$.

\begin{proposition}\label{divertN}
For a directed graph $D=(V,E)$, the following are equivalent.
\begin{enumerate}
\item[$(1)$] There exists an induced subgraph $D'$ of $D$ of which all vertices have infinite out-degree in $D'$.
\item[$(2)$] For any labeling of $V$ by $\mathbb{N}$, there exists an infinite increasing directed path in $D$.
\item[$(3)$] For any labeling of $E$ by $\mathbb{N}$, there exists an infinite increasing directed path in $D$.
\end{enumerate}
\end{proposition}

\begin{proposition}\label{thm:oriented_v_z_1}
For a directed graph $D=(V,E)$, the following are equivalent.
\begin{enumerate}
    \item[$(1)$] For every subset $V_1 \subseteq V$ where both $V_1$ and $V\setminus V_1$ are infinite, there exist $W_1\subseteq V_1$ such that all $v\in W_1$ have infinite out-degrees in $D[W_1]$.
    \item[$(2)$] For any labeling of $V$ by $\mathbb{Z}$ there is an infinite increasing directed path.
\end{enumerate}
\end{proposition}

Recall Definition~\ref{def:inf2dipath}.

\begin{proposition}\label{thm:oriented_v_z_2}
For a directed graph $D=(V,E)$, the following are equivalent.
\begin{itemize}
    \item[$(1)$] For any partition of $V$ into infinite sets $V_1, V_2$ there exists $W_1 \subseteq V_1$, $W_2 \subseteq V_2$ so that for all $v\in W_1$, $v$ has infinite out-degree in $D[W_1]$, and for all $v\in W_2$, $v$ has infinite in-degree in $D[W_2]$, and there are vertices $w_i \in W_i$ for $i=1,2$, such that $(w_2,w_1) \in E$.
    \item[$(2)$] For any labeling of $V$ by $\mathbb{Z}$ there is an infinite two-sided directed path.
\end{itemize}
\end{proposition}

\begin{proposition}\label{oriented_edgeZ1sided}
For every countable directed graph $D=(V,E)$ there exists a labeling $\phi$ of $E$ by $\mathbb{Z}$, such that there is no infinite increasing directed path in $G$.
\end{proposition}

\section{Open Problems}\label{sec:problems}

\textbf{Question.} \textit{In the statement of Theorem~\ref{thm:hyperedge}, is it true that $(2) \Leftrightarrow (3)$?}

To answer this question in the affirmative, it remains to show $(3) \Rightarrow (2)$. In our proof that $(1) \Rightarrow (2)$ in Theorem~\ref{thm:hyperedge}, we give a labeling of $H$ in which any increasing $(k-1)$-branching tree must have some branch of finite length, and this is sufficient to prevent an infinite increasing $(k-1)$-branching tree. To answer the above question by using our approach, one would need to give a labeling of $H$ in which \emph{every} branch is finite.

\textsc{Property EFIN.}
While Theorem~\ref{thm:FIN} characterizes graphs that contain arbitrarily long increasing paths under any vertex labeling, it may be of some interest to characterize graphs that have the same property with respect to \emph{edge} labeling (which we call property EFIN). Indeed, every other problem in this paper is addressed for both vertex and edge labeling. Formally, we say that a countable graph $G$ has a property EFIN if for any labeling of edges of $G$ by $\mathbb{N}$ and any $k\in \mathbb{N}$ there exists an increasing path of length $k$ in $G$ and there exists a labeling of edges with no infinite increasing path.

We consider several diverse properties, each of which we claim is related to property EFIN. Let $G$ be a countable graph.
\begin{itemize}
\item[(i)] It was proved in~\cite{Rod} that any edge-labeling of a finite graph $H$ contains an increasing path of length $\sqrt{d}(1+o(1))$, where $d$ is the average degree of $H$.  (See also~\cite{GK:73}, where a bound of $\sqrt{n}$ was proved in the case that $H=K_n$.
This lower bound of $\sqrt{n}$ was subsequently improved in~\cite{Mila:17} and~\cite{BKP}. However, the problem of improving the $\sqrt{d}$ bound for general finite graphs $H$ with average degree $d< \log n$ is still largely open.  For several interesting results related to general graphs, see~\cite{BKP}.)
\item[(ii)] For any $k$ let $T_k$ denote a tree that has height $k$ with  every non-leaf vertex having infinite degree and all leaves having height $k$. We say that $T_k$ is an infinite-branching tree of height $k$. For example, an infinite star is $T_2$.
\item[(iii)] For a graph $G$ we define $H$ -- the infinite 2-distance graph of $G$. The vertex set of $H$ is the same as the vertex set of $G$, and two vertices $u,v$ are adjacent in $H$ if and only if there are infinitely many vertices $w$, such that $uw,vw \in E(G)$.  
\end{itemize}

It is not hard to show that $G$ has property EFIN if $G$ does not contain a subgraph in which all vertices have infinite degree and $G$ satisfies at least one of the following:
\begin{itemize}
    \item[(1)] For any $d\geq 1$, $G$ contains a finite subgraph with minimal degree $d$.  
    \item[(2)] For any $k\geq 1$, $G$ contains an infinite-branching tree of height $k$.  
    \item[(3)] The infinite 2-distance graph of $G$ has infinite chromatic number.
\end{itemize} 

However we cannot prove the converse statement and pose it as open question.

{\bf Question}. \emph{What are equivalent conditions to EFIN?}


\appendix
\section{}\label{appendixB}
The four propositions in Section~\ref{sec:directed} give results for directed graphs.  The proofs of these propositions mimic the proofs of Theorem~\ref{RM}, Proposition~\ref{thm:graph_v_z_1}, Theorem~\ref{thm:graph_v_z_2}, and Proposition~\ref{edgeZ1sided} respectively, and so they are not indented to be included in the final version of this paper. They can be found in the doctoral thesis of Bradley Elliott \cite{Ell:20}.

Let $D=(V,E)$ be a directed graph where an edge $(u,v)\in E$ is oriented from $u$ to $v$. We say $D$ is in $\mathcal{V}^D_\mathbb{N}$ if for any labeling of $V$ by $\mathbb{N}$, there exists an infinite increasing directed path in $D$.

The proof of Theorem~\ref{divertN} mimics the proof of Theorem~\ref{RM}, and so we start by proving the following lemma.

\begin{lemma}\label{lemmaforvertdi}
Let $D=(V,E)$ be a directed graph and let the sets $V_i$, $i\in \mathbb{N}$, form a partition of $V$ such that the graphs $D_i = D[V_i]$ do not belong to $\mathcal{V}_{\mathbb{N}}^D$.
If for all $i\geq 1$ and each $x \in V_i$ the set
$$ \Big\{ (x,y)\in E : y\in \bigcup_{j=i+1}^\infty V_j  \Big\}$$
is finite, then $H$ does not belong to $\mathcal{V}_{\mathbb{N}}^D$.
\end{lemma}
\begin{proof}
Let $V = {v_1, v_2, \ldots}$.
For each $v\in V$, we say $h(v) = i$ if $v\in V_i$.
Let $\leq_i$ be an ordering of $V_i$ so that there is no infinite increasing directed path in $D_i$ with respect to $\leq_i$.
For any set $M\subseteq V$, define
$$R(M) = M \cup \Big\{ y\in V: \exists x\in M \text{ with } (x,y)\in E, h(x) < h(y) \Big \},$$
$$R\vert_{k} (M) = \{ x \in R(M) : h(x) \leq k\}, \text{ and}$$
$$U(M) = \Big \{ y\in V: \exists x\in M \text{ with } h(y)=h(x) \text{ and } y \leq_{h(y)} x \Big \}.$$
We say that $(U \circ R\vert_k)^1 M = U(R\vert_k(M))$ and for all $i\geq 1$, $(U \circ R\vert_k)^{i+1} M = U(R\vert_k((U \circ R\vert_k)^{i} M))$.

We define sets $Q_n, P_n \subset V$ and positive integers $k_n$ inductively.
Let $Q_0 = \emptyset$.
Now assume $Q_n$ is already given.
Let $X_{n+1} = Q_n \cup \{v_{n+1}\}$ and
$$k_{n+1} = \max\big\{ h(x) : x \in \ R(X_{n+1}) \big \}.$$
We define
$$P_{n+1} = \Big\{ y \in (U \circ R\vert_{k_{n+1}})^{k_{n+1}} X_{n+1} \Big\}.$$
Finally, let
$$Q_{n+1} = U(R(P_{n+1})) .$$
Observe that $U(R\vert_{k_{n+1}}(P_{n+1})) = P_{n+1}$.
It is clear that $Q_i$ and $P_i$ are finite for all $i$, the sequence $\{k_i\}$ is non-decreasing, and that $Q_0 \subset Q_1 \subset \cdots$ and $\bigcup_{i=0}^\infty Q_i = V$.
\begin{note}\label{note1}
If $x\in Q_i$ and $h(x)\leq k_i$ for some $i$, then $x\in P_i$.
\end{note}

We can then define an ordering $\prec$ of $V$ in the following way: we say that $x\prec y$ iff\\
\hspace*{3cm} either $x\in Q_i$, $y\not\in Q_i$ for some $i$,\\
\hspace*{3cm} or $x,y \in Q_{i+1} \setminus Q_i$ for some $i$ and $h(x) > h(y)$,\\
\hspace*{3cm} or $x,y \in Q_{i+1} \setminus Q_i$ for some $i$ and $h(x) = h(y)$ and $x\leq_{h(x)} y$.

\begin{claim}\label{claim2}
If $x\in Q_i \setminus Q_{i-1}$ for some $i$, $y\in Q_i$, and $h(x) < h(y)$, then $x\succ y$.
\end{claim}
Indeed, either $y\in Q_{i-1}$ or $y\in Q_i \setminus Q_{i-1}$, both cases of which are clear from the definition of~$\prec$.
The following lemma proves a useful property about the ordering~$\prec$.

\begin{lemma}\label{lemma1}
For an edge $(x,y)\in E$, if $h(x) \leq k_n < h(y)$ for some $n$, then $x\succ y$.
\end{lemma}
\begin{proof}
Let $x\in Q_j\setminus Q_{j-1}$ for some $j$.
Then $y\in P_{j+1}$, meaning $h(y) \leq k_{j+1}$, and $k_n < k_{j+1}$.
Thus $h(x) \leq k_j$, and Note~\ref{note1} implies $x\in P_j$.
So $y\in Q_j$, and Claim~\ref{claim2} implies $x\succ y$.
\end{proof}

Suppose now that $G$ contains an infinite increasing directed path $\{(x_i, x_{i+1})\}_{i=1}^\infty$ with respect to the ordering $\prec$.
Choose $j$ such that $h(x_1) \leq k_j$.
Then by Lemma~\ref{lemma1} we know $h(x_i) \leq k_j$ for all $i$.
Since $Q_j$ is finite, we can fix an index $i_0$ so that $x_i \not\in Q_j$ for all $i\geq i_0$.
We prove that $h(x_{i+1}) \leq h(x_i)$ for all $i\geq i_0$.
Suppose on the contrary that $h(x_{i+1}) > h(x_i)$ for some $i\geq i_0$. Let $x_i\in Q_r \setminus Q_{r-1}$ for some $r>j$, then $h(x_i)\leq k_j<k_r$. By Note~\ref{note1}, $x_i\in P_r$ and so $x_{i+1}\in Q_r$. But then, by Claim~\ref{claim2}, $x_{i+1}\prec x_i$, a contradiction.

Therefore, for some $t$, $h(x_t) = h(x_{t+1}) = h(x_{t+2}) = \cdots$.
Notice that the restriction of the ordering $\prec$ to the set $V_i$ is just the ordering $\leq_i$.
From the properties of $\leq_i$, it follows that every set $V_i$ contains only a finite number of vertices of the sequence $x_1, x_2, x_3, \ldots$.
Thus the increasing directed path cannot be infinite, and $H \not\in \mathcal{V}_{inf}$.
\end{proof}
 \begin{corollary}\label{cor1}
Let $D$ be a countable directed graph with $\deg(v)<\infty$ for all $v\in V(D)$.  Then $D \not\in \mathcal{V}_{\mathbb{N}}^D$.
\end{corollary}  \begin{proof}
Let $D_i = \{v_i\}$ so that $\vert D_i \vert =1 $ for all $i\geq 1$ and $\bigcup_{i=1}^\infty D_i = V(D)$.
\end{proof}

\begin{proof}[Proof of Theorem~\ref{divertN}]
Clearly $(1) \Rightarrow (2)$ and $(1) \Rightarrow (3)$.

To show $(2) \Rightarrow (1)$, we define the sets $M_\alpha, V_\alpha$ where $\alpha$ is an ordinal number $< \omega_1$.
We let
$$M_1 = \{v\in V : \deg_D^{out}(v) < \infty\},$$

$$V_\alpha = V\setminus \bigcup_{\beta < \alpha} M_\beta$$
$$M_\alpha = \{v\in V_\alpha : \deg_{D[V_\alpha]}^{out}(v) < \infty\}.$$
Set $s(D) = \min\{\alpha: M_\alpha = \emptyset\}$.
First we prove $\bigcup_{\alpha < \omega_1} M_\alpha = V$.
Suppose not, and let
$$V^* = V \setminus \bigcup_{\alpha < \omega_1} M_\alpha,$$
$$D^* = D[V^*],$$
and $\alpha_0 = s(D)$.
For any $v\in V^*$, if $\deg_{D^*}^{out}(v) < \infty$ then $v\in M_{\alpha}$ for some $\alpha < \omega_1$, a contradiction.
So $\deg_{D^*}^{out}(v) = \infty$ for every $v\in V^*$, which contradicts $(1)$.

Now we prove by transfinite induction on $s(D)$ that there exists a labeling of $V(D)$ by $\nn$ not containing an infinite increasing path. Lemma~\ref{lemmaforvertdi} and Corollary~\ref{cor1} prove the statement for $s(D) = \omega_0$ with $V_i = M_i$.

Assume that we proved this statement for all ordinals smaller than $s(D)$. If $s(D)$ is a successor, then $V=V_{s(D)} \bigcup _{\alpha < s(D)} M_\alpha$ is a partition of V into two sets, each of which has a labelling without an infinite increasing path. By Lemma~\ref{lemmaforvertdi}, there is a labeling of  $V(D)$ with the same property.

If $s(D)$ is a limit ordinal, then $s(D)=\lim_{n\to \infty} \alpha_n$, where $\{\alpha_i\}_1^{\infty}$ is an increasing sequence. In this case
$V=\bigcup_{n=1}^{\infty}\left(\bigcup _{\alpha_{n-1}<\alpha \leq \alpha_n} M_\alpha\right)$ is a partition of $V$ into sets $U_n=\bigcup _{\alpha_{n-1}<\alpha \leq \alpha_n} M_\alpha$, each of which has a labeling without an infinite increasing path by induction hypothesis. Again, by Lemma~\ref{lemmaforvertdi}, there is a labeling of $V(D)$ with the same property.

To show $(3) \Rightarrow (1)$, we mimic Reiterman's proof of Theorem~\ref{RM}. By (3) there exists a well-ordering $\prec$ of $V$ such that for all $v\in V$. We may assume $V=\mathbb{N}$ with the usual ordering $\leq$.
For an edge $e\in E$, let $\ell(e)$ be the vertex of $e$ that has the smaller label in the ordering $\leq$.
Conversely, $r(e)$ represents the vertex with larger label.
Let
$$L = \big\{ e\in E: \ell(e) \prec r(e) \big\}, \text{ and}$$
$$L_u = \big\{ (u,v)\in L \big\},$$
so that the two orderings $\leq$ and $\prec$ agree on the edges $L$.
Note that for all $u\in V$, there are finitely many $(u,v)\in E$ with $v \leq u$, and there are finitely many $(u,v)\in E$ with $u \prec v$.
So each set $L_u$ is finite, and the out-degree of $u$ in $L$ is finite.

\begin{claim}\label{claim3.5} 
The subgraph $D'=(V,L)$ is not in $\mathcal{V}^D_\mathbb{N}$, so there is a labeling $\phi$ of $D'$ with no infinite increasing directed path.
\end{claim}
\begin{proof}[Proof of Claim~\ref{claim3.5}]
Since the out-degree of every vertex in $V$ is finite in $L$, $L$ can be partitioned into finite sets $L_1, L_2, \ldots$ so that for all edges $(v,u)\in L_i$, any edge $(u,w)$ is in $L_j$ for some $j\leq i+1$.
Give to $L$ a labeling $\phi$ by $\mathbb{N}$ so that for all $i\in \mathbb{N}$,
$$\phi(e) < \phi(f) \text{ for all } e\in L_{2i-1}, f\in L_{2i},$$
$$\phi(f) > \phi(e) \text{ for all } f\in L_{2i}, e\in L_{2i+1}.$$
If an increasing directed path in $L$ uses an edge from any $L_{2i-1}$ or $L_{2i}$, then the path can use no edges from $L_j$ for any $j\geq 2i+1$. Such a path must be finite.\end{proof}
Now we give a labeling $\psi$ to $E\setminus L$, where the orderings $\leq$ and $\prec$ disagree.
Let
$$K_i = \big \{ (v,i), (i,v) \in E\setminus L:  v\leq i\big \}.$$
Note that sets $K_i$ are finite and $K_1, K_2, \ldots $ forms a partition of $E \setminus L$.
Suppose $\psi$ is a labeling defined on $\bigcup_{j< i} K_j $, and proceed by induction.
Define $\psi$ on $K_i$ so that
$$\text{ for all } e\in K_i \text{ and all } f\in \bigcup_{j< i} K_j, \psi(e) > \psi(f),$$
$$\text{ for all } e,f\in K_i, \psi(e) > \psi(f) \text{ iff } \ell(e) \prec \ell(f).$$
Finally, let $\gamma$ be a labeling of $E$ by $\mathbb{N}$ so that
$$\text{ for all } e,f\in L, \gamma(e) < \gamma(f) \text{ iff } \phi(e) < \phi(f),$$
$$\text{ for all } e,f\in E\setminus L, \gamma(e) < \gamma(f) \text{ iff } \psi(e) < \psi(f), \text{ and}$$
$$ \text{ for all } (u,v)\in E\setminus L, (v,w)\in L_v, \gamma((v,w)) < \gamma((u,v)),$$
which is possible because $L_v$ is a finite set.

Observe that if an edge $(u,v)\in E\setminus L$, then for all $(v,w)\in E$ with $\gamma((u,v)) < \gamma((v,w))$, $(v,w)\in E\setminus L$.
Indeed, for all $(v,w)\in L$, $(v,w)\in L_v$, so $\gamma((v,w)) < \gamma((u,v))$.
Hence an increasing directed path in $D$ with respect to $\gamma$ cannot ever use an edge of $L$ after using an edge of $E\setminus L$.

\begin{claim}\label{claim4}
If $(u,v),(v,w)\in E\setminus L$ and $\gamma((u,v)) < \gamma((v,w))$, then $\ell((u,v)) \succeq \ell((v,w))$.
\end{claim}
\begin{proof}
If $u<v<w$, then $u\succ v$ because $(u,v)\in E\setminus L$, so $\ell((u,v)) \succeq \ell((v,w))$.
If $u<v$ and $w<v$, then $(u,v),(v,w)\in K_v$, so the claim follows from the definition of $\gamma$ and $\psi$.
If $v<u$ and $v<w$, then $\ell((u,v)) = \ell((v,w))$.
Finally, it is not possible to have $w<v<u$ because this would imply, by a definition of $\psi$, that $\psi((u,v))>\psi((v,w))$ and hence $\gamma((u,v))>\gamma((v,w))$.
\end{proof}

We claim that $D$ does not contain an infinite increasing directed path with respect to $\gamma$. Suppose that $\{(v_i, v_{i+1})\}_{i=1}^\infty$ is such a path. 
This path cannot contain edges only from $L$, because $\gamma$ extends $\phi$, so there exists a $j$ such that for all $i\geq j$ edges $(v_i,v_{i+1})\in E\setminus L$.
Claim~\ref{claim4} gives that $\ell((v_i, v_{i+1})) \succeq \ell((v_{i+1}, v_{i+2})) \succeq \cdots$ for all $i\geq j$.
Since $\prec$ is a well-ordering, the set $\{\ell((v_i, v_{i+1}))\}_{j}^\infty$ contains a minimal element, so for some $k$, $\ell((v_k, v_{k+1})) = \ell((v_{k+1}, v_{k+2})) = \cdots$.
But this is impossible since no more than two edges of a path can use the same vertex, giving a contradiction.
\end{proof}

\begin{proof}[Proof of Proposition~\ref{thm:oriented_v_z_1}]
Suppose (1), and let $\phi$ be any labeling of $V$ by $\mathbb{Z}$.
Set
$$V_1 = \{v\in V: \phi(v) > 0\},$$
$$V_2 = \{v\in V: \phi(v) \leq 0\}.$$

By $(1)$ there exists a set $W_1\subseteq V_1$ so that for all $v\in W_1$, $v$ has infinite out-degree in $D[W_1]$.
Then by Theorem~\ref{divertN}, there exists an infinite increasing directed path in $D[W_1]$ and hence in $D$.

Now suppose $(1)$ does not hold.
Then there exists a partition of $V$ into infinite sets $V_1, V_2$ so every subset $W_1\subseteq V_1$ has a vertex of finite out-degree in $D[W_1]$.
Then by Theorem~\ref{divertN}, there exists a vertex-labeling of $D[V_1]$ by $\nn$ containing no infinite increasing path.
Label the vertices of $V_2$ arbitrarily by $\mathbb{Z}\setminus \nn$. Since all vertices with positive label are contained in $D[V_1]$, and since $D[V_1]$ contains no infinite increasing path, neither does $G$.

\end{proof}

\begin{proof}[Proof of Proposition~\ref{thm:oriented_v_z_2}] The proof is analogous to the proof of of Theorem~\ref{thm:graph_v_z_2}.
One implication is clear.
Suppose {(1)} and let $\phi$ be any labeling of $V$ by $\mathbb{Z}$.
We let
$$V_1 = \{v\in V: \phi(v) \geq 0\},$$
$$V_2 = \{v\in V: \phi(v) < 0\}.$$
Let $W_1$ and $W_2$ be the subsets of $V_1$ and $V_2$ respectively that are ensured by {(1)}, and let $w_1\in W_1$ and $w_2 \in W_2$ be vertices of some edge $(w_2,w_1)$.
Let $v_1 = w_1$, and for $i\geq 1$, let $v_{i+1}$ be some vertex in $W_1$ so that $(v_i,v_{i+1})\in E$ and  $\phi(v_{i+1}) > \phi(v_{i})$.
Since each $v_{i}$ has infinitely many such neighbors, $\{v_i\}_{i=1}^\infty$ forms an infinite increasing directed path in $W_1$.
Similarly let $v_0 = w_2$ and for $i\leq 0$, let $v_{i-1}$ be a vertex in $W_2$ so that $(v_{i-1},v_i)\in E$ and $\phi(v_{i-1}) < \phi(v_{i})$.
Since $(v_0,v_1)\in E$, the path formed by $\{v_i: -\infty < i < \infty\}$ is an infinite increasing two-sided directed path in $D$.

Now suppose $(1)$ does not hold. Then there are two possible cases.
\\ {\bf Case I} - there exists a partition of $V$ into infinite sets $V_1, V_2$ such that either no subset $W_1 \subseteq V_1$ induces a subgraph with all infinite out-degrees, or no subset $W_2 \subseteq V_2$ induces a subgraph with all infinite in-degrees.
In the former case, by Proposition~\ref{thm:oriented_v_z_1}, there is a labeling of $V$ by $\mathbb{Z}$ with no infinite increasing directed path.  Such a labeling clearly forbids an infinite increasing two-sided directed path. In the latter case, let $\overline{D}$ be a directed graph defined by
$$V(\overline{D}) = V(D)$$
$$E(\overline{D}) = \lbrace (u,v) : (v,u)\in E(D)\rbrace$$
Then by Proposition~\ref{thm:oriented_v_z_1}, there is a labeling $\phi$ of $V(\overline{D})$ by $\mathbb{Z}$ with no infinite increasing directed path.  Define the labeling $\psi$ of $V(D)$ by $\mathbb{Z}$ such that
$$\psi(v) = -\phi(v)$$
for all $v\in V(D)$.  Clearly $D$ contains no infinite decreasing \emph{anti-directed path} (that is, no path $\lbrace (v_{i-1},v_i): \psi(v_{i-1})< \phi(v_i),\, i\leq 1\rbrace$), it contains no infinite increasing two-sided directed path.

{\bf Case II} - there exists a partition of $V$ into infinite sets $V_1, V_2$ so that for any subsets $W_1\subseteq V_1$ with all vertices having infinite out-degree in $W_1$ and $W_2 \subseteq V_2$ with all vertices having infinite in-degree in $W_2$, there is no edge between vertices of $W_1$ and $W_2$.
For $i=1,2$, consider the family $\mathcal{W}_i$ of such sets $W_i$, where $\mathcal{W}_i$ is partially ordered by inclusion, and observe that each $\mathcal{W}_i$ contains a maximal element. Fix maximal subsets $W_1\in \mathcal{W}_i$ and $W_2\in \mathcal{W}_i$.
With $U_i = V_i\setminus W_i$ for $i=1,2$, the maximality of $W_i$ implies
\begin{enumerate}
\item[(i)] every nonempty subset $U \subseteq U_1$ contains a vertex $v\in U$ with finite out-degree in $G[U]$,
\item[(i)] every nonempty subset $U \subseteq U_2$ contains a vertex $v\in U$ with finite in-degree in $G[U]$,
\item[(iii)] for every $v\in U_1$, there are finitely many $u\in W_1$ with $(v,u)\in E$,
\item[(iv)] for every $v\in U_2$, there are finitely many $u\in W_2$ with $(u,v)\in E$.
\end{enumerate}
By (i), Theorem~\ref{divertN} implies that there is a labeling $\phi_U$ of $U_1$ by $\mathbb{N}$ that forbids infinite increasing paths in $U_1$.
By (iii) we can form a labeling $\phi$ of $V_1$ by positive integers that preserves $\phi_U$ (i.e. $\phi(u_1) < \phi(u_2)$ iff $\phi_U(u_1) < \phi_U(u_2)$ for all $u_1, u_2 \in U_1$) and where $\phi(u) > \phi(w)$ for all $u \in U_1$, $w \in W_1$  with $(u,w)\in E$.
Since $\phi$ preserves $\phi_U$, $U_1$ contains no infinite increasing directed path with respect to $\phi$.

Now follow a similar procedure for $V_2$ and $W_2$, with $U_2 = V_2 \setminus W_2$.
Here we label $V_2$ by \emph{negative} integers to prevent any infinite \emph{decreasing} anti-directed path.
Construct a labeling $\psi$ of $V_2$ by negative integers so that for any $u\in U_2$, $w\in W_2$ with $(wu)\in E$, $\psi(u) < \psi(w)$, and so that $U_2$ contains no infinite decreasing anti-directed path with respect to $\psi$.

Finally, for all $v\in V(G)$, let
\[ \gamma(v) = 
	\begin{cases} 
      \phi(v) & v\in V_1 \\
      \psi(v)+1 & v\in V_2,
   \end{cases}
\]
where we add 1 to $\psi(v)$ so that 0 is used as a label.
Note that $\gamma$ is a labeling of $V(G)$ by $\mathbb{Z}$ and $\gamma(V_1)=\mathbb{N}, \gamma(V_2)=\mathbb{Z}\setminus \mathbb{N}$.
Observe the following with respect to $\gamma$:
\[ (3)
	\begin{cases} 
      \text{there are no infinite increasing directed paths in } U_1, \\
      \text{there are no infinite decreasing anti-directed paths in } U_2, \\
      \text{there are no directed edges from } W_2 \text{ to } W_1. \\
   \end{cases}
\]
We claim that there is no infinite two-sided directed path in $G$ with respect to $\gamma$.
Suppose $\{(v_i,v_{i+1}): i\in \nn\}$ is such a path, and without loss of generality assume $v_0\in V_2$, $v_1\in V_1$.
If $v_1\in U_1$, then all vertices $v_2, v_3, \ldots$ must be in $U_1$, because $\gamma(v)>\gamma(w)$ for all $v\in E_1$ and $w\in W_1$, and so $v_1, v_2, \ldots$ is an infinite increasing directed path in $U_1$, contradicting (3).
Similarly if $v_0\in U_2$, then $v_{-1}, v_{-2}, \ldots \in U_2$, so there is an infinite decreasing anti-directed path in $U_2$, contradicting (3).
Therefore, $v_1\in W_1, v_0 \in W_2$, which contradicts (3) once again.
Therefore $G$ contains no infinite increasing two-sided directed path.

\end{proof}

\begin{proof}[Proof of Proposition~\ref{oriented_edgeZ1sided}]
This proposition follow directed from Corollary~\ref{cor2}
\end{proof}

\end{document}